\documentclass[a4paper,10pt,leqno,openright,oneside]{amsart}
\usepackage[latin1,utf8x]{inputenc}
\usepackage{amsthm,amssymb,amsmath,amsopn,amsfonts,pb-diagram,footmisc}
\usepackage{cite}
\usepackage{stmaryrd,calc,enumerate,mathrsfs,amscd,wasysym,footnote}
\usepackage{graphicx}
\usepackage{fancyhdr}
\usepackage{indentfirst}
\usepackage{geometry}
\usepackage{mathtools}
\usepackage{bbm}
\usepackage{float}

\newcommand{\CC}{\mathbb{C}}
\newcommand{\DD}{\mathbb{D}}
\newcommand{\NN}{\mathbb{N}}

\newcommand{\cB}{{\mathcal{B}}}

\renewcommand{\tilde}{\widetilde}
\newcommand{\norm}{\Vert}
\newcommand{\bea}{\begin{align}}
\newcommand{\eea}{\end{align}}
\newcommand{\beqa}{\begin{align*}}
\newcommand{\eeqa}{\end{align*}}

\newcommand{\MT}[3]{T_{#1}^{#2,#3}}

\DeclareMathOperator{\III}{I}
\DeclareMathOperator{\Rl}{Re}

\newtheorem{theorem}{Theorem}[section]

\newtheorem{lem}[theorem]{Lemma}
\newtheorem{prop}[theorem]{Proposition}

\newtheorem{conj}[theorem]{Conjecture}

\title{Generalized Integration Operators on Hardy Spaces}
\author{Nikolaos Chalmoukis}
\thanks{The first author is supported by the fellowship INDAM-DP-COFUND-2015  ''INdAM Doctoral Programme
	in Mathematics and/or Applications
	Cofunded by Marie Sklodowska-Curie Actions'', Grant 713485.}
\date{}

\address{N. Chalmoukis \\ Dipartimento di Matematica \\ Universit\'a di Bologna \\ 40127 Bologna, Italy}
\email{nikolaos.chalmoukis2@unibo.it}

\begin{document}
	\maketitle
		\begin{abstract}
		We introduce a natural generalization of a well studied integration operator acting on the family of Hardy spaces in the unit disc. We study the boundedness and compactness properties of the operator and finally we use these results to give simple proofs of a result of R\"atty\"a and another result by Cohn.
	\end{abstract}


	\section{Introduction}
	
	Let $ H^p $ me the Hardy space of analytic function in the unit disc, i.e. functions $ f $ holomorphic in $ \DD $ such that 
	\begin{equation*}
		\norm f \norm_p^p:=\frac{1}{2\pi}\sup_{0<r<1}\int_{0}^{2\pi}|f(re^{i\theta})|^p d\theta < \infty.
	\end{equation*}
	In this article $ p $ will be a positive exponent. There is a variety of linear operators acting between Hardy space which have been studied. In particular in \cite{AlemanSiskakis95}, motivated by some earlier work of Pommerenke in \cite{Pommerenke77} and Berkson and Porta in \cite{BerksonPorta78} Aleman and Siskakis introduced the integral operator $ T_g $, depending on a symbol $ g $ analytic in the unit disc, which generalized the classical Cesaro averaging operator on $ H^2 $, defined by
	\begin{equation*}
		T_gf(z)=\int_{0}^{z}f(t)g'(t)dt.
	\end{equation*}
	The operator has been studied extensively, both for its intrinsic interest and also for its applications. In particular, such an operator seems to have been introduced by Pommerenke in \cite{Pommerenke77}, in order to give a slick proof of the analytic John-Nirenberg inequality. Later on it became also evident the connection to factorization theorems for derivatives of functions in the Hardy space (see \cite{Aleman07}), and therefore to previous work of Aleksandrov and Peller on Foguel-Hankel operators in \cite{AlexandrovPeller96}. More precisely suppose we consider the bilinear operator $ T(f,g):=T_gf, \,\, T:X\times Y \to Z $, where $ X,Y,Z $ are of Banach spaces of analytic functions on the unit disc. Then the question whether or not, for every $ h\in Z $ there exists a factorization of the form $ h'=fg' $, where $ f\in X, g\in X $ and $ \norm f \norm_X + \norm g \norm_Y \lesssim \norm h \norm_Z $, can be translated  into a question about the openness of the bilinear operator $ T $. Weak factorization Theorems can be also approached in this way (see \cite{Aleman07}).
	
	 The first attempt to characterize the boundedness properties of $ T_g $ was in \cite{AlemanSiskakis95} by Aleman and Siskakis where they necessary and sufficient conditions for the boundedness of $ T_g $ on $ H^p $ for $ p\geq 1 $, while in \cite{AlemanCima01}  the result was extended for $ p>0 $. More recently, a characterization of its spectrum was found in \cite{AlemanPelaez12}. 
	
	In the present work we study a generalization of the aforementioned operator. Let us introduce the notation $ \III $ for the integration operator, i.e. 
	\begin{equation*}
	\III f(z):=\int_{0}^{z}f(t)dt.
	\end{equation*} 
	Fix now an analytic symbol $ g $ and a sequence of coefficients $ a\in\CC^{n-1} $. We shall define the generalized integration operator $ T_{g,a} $ by 
	$$ T_{g,a}f(z)= \III^n(fg^{(n)}+a_1f'g^{(n-1)}\cdots a_{n-1}f^{(n-1)}g' ).  $$
	
	The question which naturally arises is, for which symbols $ g $ this operator defines a bounded linear operator between two Hardy spaces. In this direction we prove the following Theorems.
	
	\begin{theorem}\label{p=q}
		Let $ 0<p<\infty $ and $ a\in\mathbb{C}^{n-1} $. Then, $ T_{g,a} $ is bounded from $ H^p $ to itself if and only if $ g\in BMOA $. $ T_{g,a} $ is compact if and only if $ g\in VMOA $.
	\end{theorem}

	\begin{theorem}\label{p<q}
		Let $ 0<p<q<\infty $ and $ a\in\mathbb{C}^{n-1} $. Define also $ \alpha=\frac{1}{p}-\frac{1}{q}>0 $ and $ k=\max\{l:a_l\neq 0\} $ or $ k=0 $ if $ a=0 $. Then if $ l<\alpha\leq l+1 \leq n-k $ for some $ l\in\mathbb{N} $, $ T_{g,a} $ is bounded from $ H^p $ to $ H^q $ if and only if $ g^{(l)}\in \Lambda_{a-l} $. If $ \alpha>n-k $ and $ T_{g,a} $ is bounded from $ H^p $ to $ H^q $ then $ T_{g,a} $ is the zero operator.
	\end{theorem} 
	
	\begin{theorem}\label{p>q}
		Let $ 0<q<p<\infty $ and $ a\in\mathbb{C}^{n-1} $. If $ g\in H^s $, where $ \frac{1}{s}=\frac{1}{q}-\frac{1}{p} $ then $ T_{g,a}  $ is bounded from $ H^p  $ to $ H^q $. In the special case that $ n=2 $ and $ a=0 $, if $ T_{g,a}:H^p\to H^q $ is bounded, then $ g\in H^s $.
	\end{theorem}
It is natural to state the following conjecture.
	\begin{conj}
		Is it true that if $ T_{g,a}:H^p\to H^q, 0<q<p<\infty $ is bounded then $ g $ must be in $ H^s, \frac{1}{s}=\frac{1}{q}-\frac{1}{p} $? 
	\end{conj}

	Theorems \ref{p=q} and \ref{p>q} show that the behaviour, regarding boundedness, of the operator $T_{g,a} $, in the known cases, is essentially the same with this of $ T_g $ when $ p\geq q $. In the case that $ p<q $ the operator exhibits a different behaviour because its boundedness depends on its last non zero term. It is interesting to notice that in all known cases, if $ T_{g,a} $ is bounded, every term comprising the operator is forced to be bounded as well. In other words there is no cancellation between the terms. 
	
   In the light of the manifold applications that the generalized Cesaro operator has in function theory of the Hardy space, it is natural to investigate potential applications of the generalized integration operator. Not surprisingly, using $ T_{g,\alpha} $ one can derive a result which generalizes the analytic John-Nirenberg inequality and a second Theorem which generalizes factorization of derivatives in the Hardy space. Both results exist already in the literature but our method allows for a unified and simpler approach. 
   
    The first one is of R\"atty\"a \cite{Rat07} on solutions of complex linear differential equations. 
   
   \begin{theorem}\label{DifEq}
   	Let $n\in\NN, 0<p<\infty, f_0\in H^p, G\in BMOA $ and $ g_i\in\cB, 1\leq i < n $. There exists a constant $ A>0 $ depending on $ p$  such that if $ \norm G \norm_*, \norm g_i \norm_\cB < A $, every solution of the non homogeneous linear differential equation 
   	\begin{equation*}
   	G^{(n)}f+g_1^{(n-1)}f'+\dots+g_{n-1}'f^{(n-1)}+f^{(n)}=f_0^{(n)}
   	\end{equation*} is in $ H^p $. If $ G\in VMOA $ and $g_i\in \cB_0  $, the same result holds without the restriction in the norm of $ g_i $ and $ G $.
   \end{theorem}
   
   To see why this is a Corollary of Theorem \ref{p=q} define the operator
   \begin{equation*}
   Df=\III^{n}(fG^{(n)}+f'g_1^{(n-1)}+\cdots+f^{(n-1)}g_{n-1}')
   \end{equation*} which by Theorem \ref{p=q} is bounded on $ H^p $.
   By an application of the closed graph theorem, there exists a constant $ C $ depending only on $ p$ such that 
   \begin{equation*}
   \norm D \norm_p \leq C(\norm G \norm_* + \norm g_1 \norm_\cB + \cdots + \norm g_{n-1} \norm_\cB).
   \end{equation*}
   If $ \norm G \norm_*, \norm g_i \norm_\cB < 1/(nC)  $, $ -1 $ is in the resolvent set of the operator $ D $ on $ H^p $. Let now $ f $ any solution of the above differential equation. There exists $ F\in H^p $ such that $ F-f_0 $ is a polynomial of degree less than $ n $ and $ F^{(i)}(0)=f^{(i)}(0), 0\leq i <n $. Then for some $ \tilde{f}\in H^p , D\tilde{f}+\tilde{f}=F $. But $ \tilde{f} $ and $ f $ satisfy the same differential equation with the same initial conditions, hence $ f=\tilde{f}\in H^p $. If $ G\in VMOA $ and $ g_i\in \cB_0 $, $ D $ is compact. It is easy to check that it has no eigenvalues, therefore its spectrum is the singleton $ \{0\} $, thus the result follows.

   The second result which follows from our analysis is a factorization Theorem for derivatives of functions in the Hardy space. For $ n=1 $ the result can be found in \cite{Aleman07}, and it was later generalized by Cohn for derivatives of any order in \cite{COHN2000308}. 
   
   \begin{theorem}\label{Cohn}
   	Let $ f\in H^p $, then there exist $ F\in H^p $ and $ G_n\in BMOA, n\in\mathbb{N}$ such that $ f^{(n)}=FG_n^{(n)} $.
   \end{theorem} 

The paper is organized as follows. In Section \ref{Prel} we give some preliminary results in Hardy spaces which we shall use repeatedly later.  In Section \ref{Proofs} we give the proofs of the main Theorems and  we conclude in the last Section by discussing some open questions.
\\
\paragraph*{Notation} As it is customary we shall use the notation $ A \lesssim B $, when there exists some constant $ C>0 $ independent of the parameters on which $ A, B $ such that $ A\leq CB $. If  $ A \lesssim B $ and $ B \lesssim A $ we write $ A \eqsim B $.  
	
	\subsection*{Acknowledgements} Most of the present work was done while the author was a MSc student at Lund University under the supervision of Professor Alexandru Aleman. I would like to express my gratitude to him for introducing me to the problems discussed in this article, and for his guidance and help. 


	\section{Preliminary Results}\label{Prel}

	First we will look at some variations on the Hardy-Stein identity. To do so we shall introduce the Stolz angle $ \Gamma_\sigma(e^{i\theta}) $ with vertex at $ e^{i\theta} $ and aperture $ \sigma $. That is, the interior of the convex hull of the point $ e^{i\theta} $ and the disc $ D(0,\sigma) $. Then we can define the so called square function or Lusin area function $ S_f $.
	\begin{equation*}
	S_f(\zeta)=\Big( \int_{\Gamma_\sigma(\zeta)}|f'|^2dA \Big)^{1/2}, \zeta\in \mathbb{T}.
	\end{equation*} A well known result of Fefferman and Stein (see \cite{FeffermanStein72}) states that if $ 0<p<\infty $ there exist constants $ C_1,C_2>0 $, depending only on $ \sigma $ and $ p $, such that \begin{equation*}
	C_1 \norm f \norm_p^p \leq|f(0)|^p+ \int_{\mathbb{T}}S^p_fdm \leq  C_2 \norm f \norm _p^p,
	\end{equation*} for any $ f $ analytic in $ \mathbb{D} $. Another related function is the Paley-Littlewood $ G-$function, which is defined by 
	\begin{equation*}
	G(f)(t)=\Big(\int_{0}^{1}|f'(re^{it})|^2(1-r)dr\Big)^{1/2}, \,\,\, t\in\mathbb{R}.
	\end{equation*} The Paley-Littlewood $ G-$ function enjoys the same property as the Lusin area function, i.e. there exist $ C_1,C_2>0 $, depending only on $ \sigma $ and $ p $, such that \begin{equation}\label{GFunction}
	C_1 \norm f \norm_p^p \leq |f(0)|^p+\int_{\mathbb{T}}G^p(f)dm \leq  C_2 \norm f \norm _p^p,
	\end{equation} for any $ f $ analytic in $ \mathbb{D} $. For more information about these functions the reader is referred to \cite{FeffermanStein72}. 
	
	For our purposes, we need a version of the Paley-Littlewood $ G-$function, involving only the $ n-th$ derivative of $ f $. We start with a lemma.
	
	\begin{lem}\label{EstimateForFefferman}
		Let $ f $ be an analytic function in the unit disc, then for $ z\in\mathbb{D} $ \begin{equation*}
		|f^{(n)}(z)|^2 \leq \frac{n!(n-1)!2^{2n}}{(1-|z|)^{2n}}\int_{D(z,\frac{1-|z|}{2})}|f'(\zeta)|^2dA(\zeta).
		\end{equation*}
	\end{lem}
	\begin{proof}
	 The expression of the Dirichlet integral of $ f $ in terms of its Taylor coefficients is \begin{equation*}
		\int_{\mathbb{D}}|f'(\zeta)|^2dA(\zeta) = \sum_{k\geq 0}k|a_k|^2.
		\end{equation*}Using this we have,
		\begin{equation} \label{DerivativeEstimate}
		|f^{(n)}(0)|^2 = (n!)^2|a_n|^2 
		\leq n!(n-1)! \sum_{k\geq 0}k|a_k|^2 = n!(n-1)!\int_{\mathbb{D}}|f'(\zeta)|^2dA(\zeta).
		\end{equation} 
		Now let $ z\in\mathbb{D} $ fixed and set $ r=\frac{1-|z|}{2} $. Applying (\ref{DerivativeEstimate}) to the function $ f_z(\zeta)=f(z+r\zeta), \zeta\in\mathbb{D} $, and using a change of variables: \begin{align*}
		r^{2n}|f^{(n)}(z)|^2 & \leq n!(n-1)! \int_{\mathbb{D}}r^2|f'(z+r\zeta)|^2dA(\zeta) \\
		& = n!(n-1)! \int_{D(z,\frac{1-|z|}{2})}|f'(\zeta)|^2dA(\zeta),
		\end{align*} which gives the desired inequality.
	\end{proof}
	
	Now, let $ f $ analytic in $ \mathbb{D} $. We define the Paley-Littlewood $ G_k-$function of order $ k $ to be  \begin{equation}
	G_k(f)(t)=\Big(\int_{0}^{1}|f^{(k)}(re^{it})|^2(1-r)^{2k-1}dr\Big)^{1/2}.
	\end{equation}
	\begin{prop}\label{VariationPalley}
		Let $p>0 $ and $ f $ analytic in $ \mathbb{D} $ with $ f^{(i)}(0)=0, 0\leq i<k $ then, for any $ k\in\mathbb{Z}_+ $, there exist constants $ C_1, C_2 $ depending only on $ p $ and $ k $ such that \begin{equation*}
		C_1\norm f \norm_{H^p} \leq \norm G_k(f) \norm_{L^p} \leq C_2 \norm f \norm_{H^p}.
		\end{equation*}
		\begin{proof}
			We prove the left inequality for $ k=2 $, the general case follows by induction. First, let $ f $ be analytic in an open set containing the closure of $ \mathbb{D} $. Then for $ t\in\mathbb{R} $ fixed
			\begin{align*}
			G_1^2(f)(t)=\int_{0}^{1}|f'(re^{it})|^2(1-r)dr &= \frac{1}{2}\int_{0}^{1}\frac{\partial |f'(re^{it})|^2}{\partial r} (1-r)^2dr \\
			&= \int_{0}^{1}\Rl(e^{it}f''(re^{it})\bar{f'}(re^{it}))(1-r)^2dr \\
			&\leq \int_{0}^{1}|f''(re^{it})||f'(re^{it})|(1-r)^2dr \\ 
			&\leq G_1(f)(t)G_2(f)(t),
			\end{align*}by Cauchy-Schwarz. Dividing through by $ G_1(f)(t) $ we have the result for $ f $ analytic in a larger disc. Now if $ f $ is an arbitrary analytic function, fix $ 0<\rho<1 $ and consider the dilations $ f_\rho(z)=f(\rho z) $. Then, \begin{align*}
			G_1(f_\rho)(t) &\leq \rho^4\int_{0}^{1}|f''(r\rho e^{it})|^2(1-r)^3dr \\
			&\leq \rho\int_{0}^{\rho}|f''(ue^{it})|^2(1-u)^3du.
			\end{align*}Then by taking liminf in both sides as $ \rho\to 1^- $ and applying Fatou's lemma on the left and monotone convergence on the right we conclude that $ G_1(f)\leq G_2(f) $. The desired inequality then follows by (\ref{GFunction}).
			
			To prove the right inequality, we will use Lemma \ref{EstimateForFefferman}. First note that for any $ z=re^{i\theta}\in\mathbb{D},  D(re^{i\theta},\frac{1-r}{2}) \subset \Gamma_{\frac{1}{2}}(e^{i\theta}) $, which, together with Lemma \ref{EstimateForFefferman}, justifies the following calculation \begin{align*}
			G^2_{k}(f)(\theta) &=\int_{0}^{1}|f^{(k)}(re^{i\theta})|^2(1-r)^{2k-1}dr \\
			&\leq C_k \int_{0}^{1}\int_{D(re^{i\theta},\frac{1-r}{2})}|f'(\zeta)|^2dA(\zeta)(1-r)^{-1}dr \\ 
			&= C_k \int_{\Gamma_{\frac{1}{2}}(e^{i\theta})} |f'(\zeta)|^2 \int_{0}^{1} \chi_{D(r,\frac{1-r}{2})}(\zeta)(1-r)^{-1}drdA(\zeta).
			\end{align*}
			It is routine to check that if $ |\zeta|<\frac{3r-1}{2} $ or $ |\zeta|>\frac{1+r}{2} $ then $ \zeta \not\in D(r,\frac{1-r}{2}) $. Hence 
			\begin{equation*}
			\int_{0}^{1} \chi_{D(r,\frac{1-r}{2})}(\zeta)(1-r)^{-1}dr \leq \int_{2|\zeta|-1}^{\frac{2|\zeta|+1}{3}}(1-r)^{-1}dr = \log 3
			\end{equation*} Therefore we have proven that $ G_k^2(f)(\theta) \leq S_f^2(\theta) $ and the estimate follows by Fefferman-Stein's theorem.
		\end{proof}
	\end{prop}

	We conclude this section with a Lemma of linear algebra. For $ \gamma>0 $ we use the notation $ (\gamma)_0=1, (\gamma)_k=\gamma(\gamma+1)\cdots(\gamma+k-1) $ for $ k\geq 1 $.
	
	\begin{lem}\label{LinearAlg}
		Suppose that $ f_0, f_2, \dots f_{n-1} $ are complex valued functions on the unit disc (not necessarily analytic), such that for any $ \gamma\in \mathbb{R} $ sufficiently large there exists $ C_\gamma>0 $ such that 
		\begin{equation}
		\Big| \sum_{k=0}^{n-1}f_k(z)(\gamma)_k \Big| \leq C_\gamma, z\in\mathbb{D}.
		\end{equation}Then all $ f_k $ are bounded.
	\end{lem}
	
	\begin{proof}
		Choose distinct $ \gamma_0, \gamma_2, \dots \gamma_{n-1} $ sufficiently large. It is a tedious but standard calculation that 
		\begin{equation*}
		\det\begin{bmatrix}
		(\gamma_0)_1 & (\gamma_0)_2 & \ldots
		& (\gamma_0)_{n-1} \\
		(\gamma_1)_1 & (\gamma_1)_2 & \ldots
		& (\gamma_1)_{n-1} \\
		\vdots & \vdots & \ddots
		& \vdots \\
		(\gamma_{n-1})_1 & (\gamma_{n-1})_2 & \ldots
		& (\gamma_{n-1})_{n-1}
		\end{bmatrix} = \prod_{0\leq i<j < n}(\gamma_j-\gamma_i) \neq 0.
		\end{equation*} In other words, the vectors $ \Gamma_k=((\gamma_k)_0,\dots,(\gamma_k)_{n-1}), k=0,1,\dots n-1 $ form a basis of $ \mathbb{R}^n $. Therefore for a fixed $ k, 0\leq k <n $, there exist $ r_0,\dots,r_{n-1}\in\mathbb{R} $, such that \begin{equation*}
		(0,\dots,0,1,0,\dots,0)=\sum_{i=0}^{n-1}r_i\Gamma_i,
		\end{equation*}where the vector on the left of the above equation has all components, except for the $ k-$th,  equal to zero.
		Therefore, 
		\begin{align*}
		f_k(z) &= \sum_{j=0}^{n-1}f_j(z)\Big(\sum_{i=0}^{n-1}r_i(\gamma_i)_j\Big) \\
	&	= \sum_{i=0}^{n-1}r_i\sum_{j=0}^{n-1}f_j(z)(\gamma_i)_j.
		\end{align*} Hence, by the our assumptions, 
		\begin{equation*}
		|f_k(z)| \leq \sum_{i=0}^{n-1}|r_i|C_{\gamma_i}.
		\end{equation*}
	\end{proof}


	\section{Proofs of the Main Theorems}\label{Proofs}

In order to understand the behaviour of the operator $ T_{g,a} $ it will be useful to consider the operators $ \MT{g}{n}{k} $ defined for an analytic function $ g $ and natural numbers $ n,k $ such that $ 0 \leq k < n $, by the formula
\begin{equation*}
\MT{g}{n}{k}f=\III^{n}(f^{(k)}g^{(n-k)}).
\end{equation*}  


\subsection{Boundedness}
\subsection*{Proof of Theorem \ref{p=q}} The main step to prove the sufficiency part of Theorem \ref{p=q}, is the following proposition. 
\begin{prop} \label{Bloch}
	Let $ n,k\in\mathbb{Z}_+, k < n$, and $ g\in\cB $. Then the operator $ \MT{g}{n}{k}:H^p\to H^p $ is bounded.
\end{prop}

\begin{proof}
	Let $ f $ be an analytic function in an open set containing the closure of the unit disc. We can assume without loss of generality that $ f^{(i)}(0)=0, 0\leq i < k $ because it is readily checked that $ \MT{g}{n}{k} $ maps the set of polynomials in $ H^p $. Then we have that \begin{align*}
	G_n( \MT{g}{n}{k}f)(t) &= \Big( \int_{0}^{1}|f^{(k)}(re^{it})|^2|g^{(n-k)}(re^{it})|^2(1-r)^{2n-1}dr\Big)^{1/2}\\
	&\lesssim \norm g \norm_{\cB}\Big( \int_{0}^{1}|f^{(k)}(re^{it})|^2(1-r)^{2k-1}dr\Big)^{1/2} \\
	&\lesssim \norm g\norm_{\cB}  G_k(f)(t).
	\end{align*} The result follows immediately from Proposition \ref{VariationPalley}, and a density argument.
\end{proof}
To prove sufficiency in Theorem  \ref{p=q}, notice that \begin{equation}\label{Tgmakrinari}
T_g=\MT{g}{n}{0}+\sum_{k=1}^{n-1} {{n-1}\choose{k}} \MT{g}{n}{k}.
\end{equation} Therefore if $ g\in BMOA \subset \cB $, by Proposition \ref{Bloch} and the fact that $ T_g $ is bounded when $ g\in BMOA $ (see \cite{AlemanCima01}), $ \MT{g}{n}{0} $ is bounded as well. Hence $ T_{g,a} $ is bounded as well.

We suppose that $ T_{g,a}:H^p\to H^p $ is bounded. It suffices to prove that in that case $ g\in\cB $. this together with Proposition \ref{Bloch} and equation (\ref{Tgmakrinari}) imply that $ T_g $ is bounded, and therefore $ g\in BMOA $, because the result is known to hold for $ n=1 $ (see \cite{AlemanCima01}).

For $ \lambda\in\mathbb{D}, \gamma>1/p $ set \begin{equation*}
f_{\lambda,\gamma}(z)=\frac{(1-|\lambda|^2)^{\gamma-1/p}}{(1-\bar{\lambda}z)^\gamma}.
\end{equation*} Hence, 
\begin{equation}\label{TestFunctions}
f_{\lambda,\gamma}^{(k)}(\lambda)=\frac{(\gamma)_k\bar{\lambda}^k}{(1-|\lambda|^2)^{k+1/p}},\,\, k\in\mathbb{N},
\end{equation}  
where $ (\gamma)_0=1, (\gamma)_k=\gamma(\gamma+1)\cdots(\gamma+k-1) $. Also there exists a positive constant $ C_\gamma $ such that $ \norm f_{\lambda,\gamma}\norm_p \leq C_\gamma, \lambda\in\mathbb{D} $. Then the growth estimate for $ H^p $ functions gives \begin{align*}
\frac{C_\gamma}{(1-|\lambda|^2)^{n+1/p}} &\geq |(T_{g,a}f_{\lambda,\gamma})^{(n)}(\lambda)| \\
&=   \Big| \sum_{k=0}^{n-1}\frac{a_k\bar{\lambda}^k(\gamma)_k}{(1-|\lambda|^2)^{k+1/p}}g^{(n-k)}(\lambda) \Big|.
\end{align*} Or, rearranging the inequality, 
\begin{equation}\label{grammikablegmeno}
\Big| \sum_{k=0}^{n-1}g^{(n-k)}(\lambda)(1-|\lambda|^2)^{(n-k)}a_k\bar{\lambda}^k(\gamma)_k \Big| \leq C_\gamma.
\end{equation} By applying Lemma \ref{LinearAlg} we can infer that $ \sup_{\lambda\in\mathbb{D}}|g^{(n)}(\lambda)|(1-|\lambda|^2)^n < \infty $, i.e. $ g\in \cB $.

At this point we are able to give the promised simple proof of Proposition \ref{Cohn}. We shall need the next lemma which is of some interesting on its own right.

\begin{lem}
	If $ g $ is in $ BMOA $ the operator $ \MT{g}{n}{n-1} $ is a bounded linear operator from $ \cB $ to $ BMOA $.
\end{lem}
\begin{proof} Let $ f\in \cB $ and set $ F=\MT{g}{n}{n-1}f $. Then for an arbitrary $ h\in H^2 $
	\begin{equation*}
	T_{F,0}h=\III^{n}(f^{(n-1)}g'h) 
	=\MT{f}{n}{1}T_gh.
	\end{equation*}But $ T_gh\in H^2 $ since $ g\in BMOA $ and by Proposition \ref{Bloch}, $ \MT{f}{n}{1} $ is bounded on $ H^2 $. Hence $ T_{F,0} $ is bounded on $ H^2 $ as well. By the necessity part of Theorem $ \ref{p=q} $ we get that $ F\in BMOA $.
\end{proof}

\begin{proof}[Proof of Proposition \ref{Cohn}]
	Let $ f\in H^p $, and define  $ F $  by \begin{equation*}
	F(z)=\Big(\int_{\mathbb{T}}\frac{\zeta+z}{\zeta-z}|f(\zeta)|^{p/2}dm(\zeta)\Big)^{2/p}.
	\end{equation*} Since $ F $ has positive real part, $ \log F\in BMOA $. Now we proceed by induction. The statement is trivial for $ n=0 $, hence suppose that it is true for $ n-1 $, \begin{align*}
	f^{(n)} &=(FG_{n-1}^{(n-1)})' \\
	&=F((\log F)'G_{n-1}^{(n-1)}+G_{n-1}^{(n)}) \\
	&=F(\MT{\log F}{n}{n-1}G_{n-1}+G_{n-1})^{(n)}.
	\end{align*} But $ \log F\in BMOA \subset \cB $, hence, by the previous lemma $\MT{\log F}{n}{n-1}G_{n-1} \in BMOA$.
\end{proof}

\subsection*{Proof of Theorem \ref{p<q}} As before we prove a seemingly stronger statement.
\begin{prop}
	Let $ 0<p<q<\infty, n\in\mathbb{N} $ and $ 0\leq k<n $ fixed. Then set $ \alpha=\frac{1}{p}-\frac{1}{q} $. If $ l<\alpha \leq l+1\leq n-k $ for some $ l\in\mathbb{N} $ and $ g^{(l)}\in \Lambda_{\alpha-l} $, then $ \MT{g}{n}{k} $ is bounded from $ H^p $ to $ H^q $.
\end{prop}
\begin{proof}For $ n=1 $ the operator is just $ T_g $, and the statement of the Theorem reduces to the known result about $ T_g $ (see \cite{AlemanCima01}). Suppose now that the statement is true for some $ n\geq1 $ and proceed by induction. We distinguish two cases. If $ n-k<\alpha\leq n+1-k $, take an arbitrary $ f\in H^p $ and write $ f^{(k)}=FG^{(k)} $ for some $ F\in H^p $ and $ G\in BMOA $. The following identity is true, 
	\begin{equation*}
	\MT{g}{n+1}{k}f=\III^{(n-k)}\MT{g}{k+1}{1}T_{g^{(n-k)}}F.
	\end{equation*} 

By the assumption that $ g^{(n-k)}\in \Lambda_{\alpha-n+k} $ and by the case $ n=1 $ we conclude that $ T_{g^{(n-k)}}$ is  bounded from $H^p $ to $ H^{p'}, $ where $ \frac{1}{p'}=\frac{1}{q}+n-k $. It follows, using Proposition \ref{Bloch} and the boundedness properties of $ \III $ that in this case $ T_{g,a}:H^p\to H^q $ is bounded. Suppose now that $ \alpha \leq n-k $. Assuming without loss of generality that $ g $ has sufficient zero multiplicity at the origin. Integrating by parts we get
	\begin{equation*}
	\MT{g}{n+1}{k}f=\MT{g}{n}{k}f-\MT{g}{n+1}{k+1}f.
	\end{equation*}

The first term on the right hand side is bounded by the induction hypothesis. To prove the boundedness of the second term, factorize $ f^{(k+1)} $ as in Proposition \ref{Cohn}. Then,
	\begin{equation*}
	\MT{g}{n+1}{k+1}f=\MT{G}{n+1}{n-k}\MT{g}{n-k}{0}F.
	\end{equation*}  The result follows by the induction hypothesis and Theorem \ref{p=q} applied to the operator $ \MT{G}{n+1}{n-k} $.
\end{proof}

The proof of necessity in Theorem \ref{p<q} is similar to that of the necessity part of Theorem \ref{p=q}. Consider again the family of test functions $ f_{\lambda,\gamma} $. As before, \begin{equation*}
\Big| \sum_{k=0}^{n-1}\frac{a_k\bar{\lambda}^k(\gamma)_k}{(1-|\lambda|^2)^{k}}g^{(n-k)}(\lambda) \Big| \leq \frac{C_\gamma}{(1-|\lambda|^2)^{\frac{1}{q}-\frac{1}{p}+n}},
\end{equation*}  by applying again Lemma \ref{LinearAlg} we can separate the previous condition to the following.
\begin{equation}\label{estimatemea}
|g^{(n-k)}(\lambda)|\lesssim \frac{1}{(1-|\lambda|^2)^{\frac{1}{q}-\frac{1}{p}+n-k}},\,\, \lambda\in\mathbb{D},
\end{equation} where $ k=\max\{l:a_l\neq 0\} $. Since $ a_k\neq 0 $ it follows immediately that $ g^{(l)}\in \Lambda_{\alpha-l} $ if $ \alpha\leq n-k $ and $ g^{(n-k)}=0 $  if  $ \alpha>n-k $.

\subsection*{Proof Theorem \ref{p>q}}The proof of sufficiency in Theorem \ref{p>q} follows the same pattern. Again we will prove by induction on $ n $ that for $ 0\leq k <n, \MT{g}{n}{k}$ is bounded from $ H^p $ to $ H^q $. As usual write $ f^{(k)}=FG_k^{(k)} $  as in Proposition \ref{Cohn} and use  the same recursive formula, i.e.,
\begin{equation*}
\MT{g}{n+1}{k}f=\MT{g}{n}{k}f-\MT{G_k}{n+1}{n-k}\MT{g}{n-k}{0}F.
\end{equation*} which by the induction hypothesis and Theorem \ref{p=q} proves our claim.

The proof of necessity in Theorem \ref{p>q} is more involved and is based on the following lemma which can be proved for arbitrary $ a\in\mathbb{C} $.

\begin{lem} \label{ToLemma}
	For every $ q>0 $ there exists a positive constant $ C=C(q)>q $ such that if $p>C(q)$  and $ T_{g,a}:H^p \to H^q $ then $ g\in H^s, \frac{1}{s}=\frac{1}{q}-\frac{1}{p} $.
\end{lem}
\begin{proof}Assume that $ g(0)=g'(0)=0 $, multiple integrations by parts allow us to rewrite the formula expressing $ T_{g,a} $ as 
	\begin{align}\label{analysitelesti}
	T_{g,a}f &= fg+(1-a)\int\int f''g+(a-2)\int f'g \\
	 & = fg+(1-a)\int\int\big( \frac{f''}{f}-\frac{f'^2}{f^2}\big)gf+(1-a)\int\int\frac{f'^2}{f^2}gf+(a-2)\int\frac{f'}{f}gf, 
	\end{align}for any $ f\in H^p $, where we used the convention 
	\begin{equation*}
	\int f = \int_{0}^{z}f(t)dt.
	\end{equation*} Now let $ \epsilon>0 $ and $ G = \mathcal{H}\big(\frac{|g|^\alpha}{(1+\epsilon|g|)^\alpha}\big)^\beta$, where $ \alpha $ and $ \beta $ are positive constants to be specified later and $ \mathcal{H} $ denotes the Herglotz transform, i.e. 
	\begin{equation*}
	G(z)=\Big( \int_{\mathbb{T}} \frac{|g(\zeta)|^\alpha}{(1+\epsilon|g(\zeta)|)^\alpha} \frac{\zeta+z}{\zeta-z}  dm(\zeta) \Big)^\beta.
	\end{equation*}
	$ G $ has positive real part , hence $ \norm \log G\norm_* \lesssim \beta $ and also $ G\in H^\infty $. If we assume that $ \beta p>1 $ then by M. Riesz's theorem \cite[Theorem 4.1]{Duren00}
	\begin{align}
	\norm G \norm_p^p & = \int_{\mathbb{T}}\mathcal{H}\big(\frac{|g|^\alpha}{(1+\epsilon|g|)^\alpha}\big)^{\beta p} dm \\
	& \lesssim \int_{\mathbb{T}}\Big( \frac{|g|}{1+\epsilon|g|}  \Big)^{\alpha\beta p}dm. \label{estimate18}
	\end{align}
	The last estimate we need for $ G $ is that $ |G| \geq (\Rl G^{1/\beta})^\beta = \Big( \frac{|g|}{1+\epsilon|g|}  \Big)^{\alpha\beta }. $ Which gives \begin{equation} \label{estimate19}
	\norm Gg \norm_q^q = \int_{\mathbb{T}} |Gg|^qdm \geq \int_{\mathbb{T}} \Big( \frac{|g|}{1+\epsilon|g|} \Big)^{(\alpha\beta+1)q}.
	\end{equation} With these preliminaries we are going to estimate $ \norm T_{g,a}G \norm_q $.  First we estimate separately the last three terms in the right hand side of (\ref{analysitelesti}). 
	\begin{align}\label{protiektimisi}
	\Big\norm \int\int (\log G)''Gg \Big\norm_q &\leq C_1(q)\norm \log G \norm_* \norm Gg\norm_q \\ 
	&\leq C_2(q)\beta \norm Gg \norm_q,
	\end{align}by the sufficiency part of Theorem \ref{p=q}.
	\begin{align*}
	\Big\norm \int\int\frac{G'^2}{G^2}gG \Big\norm_q & = \Big\norm \int\int (\log G)' \Big( \int (\log G)'gG \Big)'\Big\norm_q \\
	&\leq C_3(q) \beta^2 \norm gG \norm_q.
	\end{align*}
	\begin{equation*}
	\Big\norm \int \frac{G'}{G}gG \Big\norm_q \leq C_4(q)\beta \norm gG \norm_q.
	\end{equation*} Then (\ref{analysitelesti}) and the boundedness of $ T_{g,a} $ give \begin{equation*}
	\norm Gg \norm_q(1-C_5(q)\beta-C_6(q)\beta^2) \lesssim \norm T_{g,a} \norm_{p,q}\norm G\norm_p.
	\end{equation*} Furthermore, if $ \beta C_5(q)+\beta^2C_6(q) <1 $ (or equivalently $ \beta<C(q)$ where $ C(q)$ is a continuous function of $C_5,C_6 $)
	we arrive at \begin{equation*}
	\norm Gg \norm_q \leq C_{q,\beta} \norm T_{g,a} \norm_{p,q} \norm G \norm_p,
	\end{equation*} which together with estimates (\ref{estimate18}) and (\ref{estimate19}) give 
	\begin{equation*}
	\Big( \int_{\mathbb{T}} \big(\frac{|g|}{1+\epsilon|g|}\big) ^{(\alpha\beta+1)q}   \Big)^{1/q} \leq \tilde{C}_{q,\beta} \norm T_{g,a} \norm_{p,q} \Big( \int_\mathbb{T} \big( \frac{|g|}{1+\epsilon|g|}\big)^{\alpha\beta p}  \Big)^{1/p}.
	\end{equation*} Now, choose $ \alpha $ such that $ (\alpha\beta+1)q=\alpha\beta p $ and note that in this case $ (\alpha\beta+1)q=\frac{1}{s} $. Hence 
	\begin{equation*}
	\Big( \int_{\mathbb{T}}\big( \frac{|g|}{1+\epsilon|g|} \big)^sdm \Big)^{1/s} \leq \tilde{C_{q,\beta}} \norm T_{g,a} \norm_{p,q}.
	\end{equation*} Fatou's lemma then gives $ g\in H^s $.
\end{proof}

We can now prove the necessity part in Theorem \ref{p>q}. Let $ T_{g,0}:H^p\to H^q, p>q $ be bounded. And set $ \frac{1}{s}=\frac{1}{q}-\frac{1}{p} $. Note that by the Riesz-Thorin Theorem (see for example \cite[Remark 2.2.5]{Zhu90}) we can choose the constant $ C=C(q) $ in the previous lemma to stay bounded if $ 0<\epsilon<q<1/\epsilon $ for some $ \epsilon>0 $. Therefore \begin{equation*}
C_0=\sup_{\frac{1}{s}\leq q\leq \frac{1}{s}+\frac{1}{p}}C(q) < \infty.
\end{equation*}Pick a natural number $ n $ such that $ np>C_0 $ and define $ p'=np $. Then if $ f_i\in H^{p'} $ \begin{equation*}
T_{g,0}(f_1f_2\cdots f_n)=P_{T_{g,0}(f_2f_3\cdots f_n),0}(f_1).
\end{equation*} Keeping $ f_2,f_3,...,f_n $ fixed and applying the previous lemma  to the operator $T_{T_{g,0}(f_2f_3\cdots f_n),0}$ we have that $ T_{g,0}(f_2f_3\cdots f_n)\in H^{q_1}, \frac{1}{q_1}=\frac{1}{q}-\frac{1}{p'} $. Continuing inductively we arrive at $ g\in H^{q_n}, \frac{1}{q_n}=\frac{1}{q}-\frac{n}{p'}=\frac{1}{s} $.

\subsection{Compactness}\label{Compactness}

Now, we are going to prove the compactness part of Theorem \ref{p=q}. The proof is similar to the one for boundedness, therefore,  first we prove the result corresponding to Proposition \ref{Bloch} for the operators $ \MT{g}{n}{k} $ when $ g $ is in the little Bloch space.

\begin{prop}\label{CompactnessBloch}
	Let $ 0<p<\infty $ and $ g\in\cB_0 $. Then for $ n>1, 1<k<n $  the operator $ \MT{g}{n}{k} $ is a compact operator from $ H^p $ to itself. If $ g\in VMOA, \MT{g}{n}{0} $ is compact from $ H^p $ to itself. 
\end{prop}
\begin{proof}
	By integration by parts one can write $ \MT{g}{n}{k} $ in the following form 
	\begin{equation*}
	\MT{g}{n}{k}f=c_1I^{(k)}M_{g^{(k)}}f+c_2I^{(k+1)}M_{g^{(k+1)}}f+ \dots +c_{n-k}\III^{n}M_{g^{(n)}}f. 
	\end{equation*}Where $ M_\phi $ is the multiplication operator defined by 
	\begin{equation*}
	M_\phi f(z)=f(z)\phi(z).
	\end{equation*} 
	If $ g $ is polynomial, $ M_{g^{(i)}}, $ for $ k\leq i \leq n $ is bounded on $ H^p $. It is a well known fact (see \cite{AlemanCima01}) that $ \III $ is compact on $ H^p $, hence if $ g $ is polynomial $ \MT{g}{n}{k} $ is a compact operator. Suppose now that $ g\in\cB_0 $. Then there exists a sequence of polynomials $ g_n $ converging in the Bloch norm to $ g $. Therefore by an application of the open mapping Theorem, 
	\begin{equation*}
	\norm \MT{g}{n}{k}-\MT{g_m}{n}{k} \norm_p \leq C_{p}\norm g_m-g \norm_\cB,
	\end{equation*}i.e. $ \MT{g}{n}{k} $ is compact as the norm limit of compact operators.
	
	The proof of the second part is identical, since $ VMOA $ is the closure of the set of polynomials in  $ BMOA $.
\end{proof}

By Proposition \ref{CompactnessBloch} and the fact that $ VMOA \subset \cB_0 $, it follows immediately that $ T_{g,a} $ is compact if $ g
\in VMOA $.

To show necessity it suffices to show that $ g\in\cB_0 $, because then by Proposition \ref{CompactnessBloch} and formula (\ref{Tgmakrinari}) it follows that $ T_g $ is compact and therefore $ g\in VMOA $. Let $ f_{\lambda,\gamma} $ be defined by (\ref{TestFunctions}). For fixed $ \gamma $, $ f_{\lambda,\gamma} $ converges to zero on compact sets as $ |\lambda |\to 1^- $. The compactness of $ T_{g,a} $ then implies that $ T_{g,a}f_{\lambda,\gamma}  $ converges strongly to $ 0 $.

Using the same estimates as in the proof of necessity of Theorem \ref{p=q}, we get the following inequality.
\begin{equation*}
\Big| \sum_{k=0}^{n-1}g^{(n-k)}(\lambda)(1-|\lambda|^2)^{(n-k)}a_k\bar{\lambda}^k(\gamma)_k \Big| \leq \norm T_{g,a}f_{\lambda,\gamma} \norm_p.
\end{equation*}Again, by applying Lemma \ref{LinearAlg} we can separate the above estimate.
\begin{equation*}
|g^{(n)}(\lambda)|(1-|\lambda|^2)^n \lesssim \sum_{i=1}^{n}c_i \norm T_{g,a}f_{\lambda,\gamma_i} \norm_p,
\end{equation*}for some positive constants $ c_i $. This last inequality gives the desired result, since the right part converges to zero as $ |\lambda|\to 1^- $.

\bibliography{RawFormatLibrary}
\bibliographystyle{ieeetr}
\end{document}